\documentclass[12pt]{article}
\usepackage[a4paper,margin=1.0in]{geometry}

\usepackage[colorlinks,citecolor=magenta,linkcolor=black]{hyperref}
\pdfpagewidth=\paperwidth \pdfpageheight=\paperheight
\usepackage{amsfonts,amssymb,amsthm,amsmath,eucal,tabu,url}
\usepackage{pgf}
 \usepackage{array}
 \usepackage{tikz-cd}
 \usepackage{pstricks}
 \usepackage{pstricks-add}
 \usepackage{pgf,tikz}
 \usetikzlibrary{automata}
 \usetikzlibrary{arrows}
 \usepackage{indentfirst}
\usepackage{tabularx} 

\theoremstyle{plain}
\newtheorem{thm}{Theorem}[section]
\newtheorem{theorem}[thm]{Theorem}
\newtheorem*{theoremA}{Theorem A.1}

\newtheorem{proposition}[thm]{Proposition}

\theoremstyle{definition}
\newtheorem{definition}[thm]{Definition}

\newtheorem{example}[thm]{Example}

\newtheorem{thevarthm}[thm]{\varthmname}

\newenvironment{varthm*}[1]{\trivlist\item[]{\bf #1.}\it}{\endtrivlist}

\newtheorem{custom}{{\rm Theorem}}

\renewcommand\geq{\geqslant}

\renewcommand\leq{\leqslant}

\newcommand\be{\begin{eqnarray*}}
\newcommand\ee{\end{eqnarray*}}

\newcommand\newop[2]{\def#1{\mathop{\rm #2}\nolimits}}
\newop\log{log}
\newop\ord{ord}
\newop\Gal{Gal}
\newop\SL{SL}
\newop\Bl{Bl}
\newop\mult{mult}
\newop\mass{mass}
\newop\div{div}
\newop\codim{codim}
\newop\sing{sing}
\newop\vdim{vdim}
\newop\edim{edim}
\newop\Ass{Ass}
\newop\size{size}
\newop\reg{reg}
\newop\satdeg{satdeg}
\newop\supp{supp}
\newop\Neg{Neg}
\newop\Nef{Nef}
\newop\Nefh{Nef_H}
\newop\Eff{Eff}
\newop\Zar{Zar}
\newop\MB{MB}
\newop\MBxC{MB\mathit{(x,C)}}
\newop\NnB{NnB}
\newop\Bigg{Big}
\newop\Effbar{\overline{\Eff}}

\def\keywordname{{\bfseries Keywords}}%
\def\keywords#1{\par\addvspace\medskipamount{\rightskip=0pt plus1cm
\def\and{\ifhmode\unskip\nobreak\fi\ $\cdot$
}\noindent\keywordname\enspace\ignorespaces#1\par}}
\def\subclassname{{\bfseries Mathematics Subject Classification
(2020)}\enspace}
\def\subclass#1{\par\addvspace\medskipamount{\rightskip=0pt plus1cm
\def\and{\ifhmode\unskip\nobreak\fi\ $\cdot$
}\noindent\subclassname\ignorespaces#1\par}}

\begin{document}
\title{On maximizing curves of degree $7$}
\author{Izabela Czarnota}
\date{\today}
\maketitle

\thispagestyle{empty}
\begin{abstract}
In the present paper we investigate the question concerning the existence of maximizing curves of degree $7$ with some prescribed ${\rm ADE}$ singularities. We give a result proving the non-existence of such maximizing septics and we provide new examples of conic-line arrangements with some ${\rm ADE}$ singularities that are free but not maximizing.
\keywords{conic-line arrangements, simple singularities}
\subclass{14N25, 32S25, 51B05, 51A45}
\end{abstract}
\section{Introduction}
In the present note we study combinatorial constraints on maximizing reduced plane curves in the complex projective plane. Such curves are required to have only ${\rm ADE}$ singularities and, by recent results due to Dimca and Pokora \cite{DimcaPokoraMaximizing}, these are strictly connected with the world of free plane curves. It is worth recalling that free plane curves are intensively studied by many authors and there are many intriguing problems motivated by geometrical properties of free plane curves - for more details regarding this subject we refer to the very recent survey by Dimca \cite{SurDim}. Our work is mainly motivated by the fact that maximizing curves in odd degrees are apparently very rare objects. For example, we know of only one example of such a maximizing curve in degree $9$, see \cite[Example 6.4]{DIPS}, and we are not aware of examples of such curves in odd degrees $>9$. From this perspective, in order to better understand the geometry of such curves, we need to squeeze out some geometric conditions restricting their existence. In the present paper we focus on the case when our curves are of degree $7$. In such a situation, we can try to narrow the domain of our search for such maximizing curves. Our main result can be formulated as follows.

 \begin{custom}[see Theorem \ref{char}]
There does not exist any reduced plane curve $C$ of degree $7$ with singularities of types $A_{1}, A_{3}, A_{5}, A_{7}, D_{4}, D_{6}$ that is maximizing.
 \end{custom}
 
 Then we focus on the case of conic-line arrangements admitting singularities of types $A_{1}, A_{3}, A_{5}, A_{7}, D_{4}, D_{6}$ and $D_{8}$. There are exactly $103$ admissible weak-combinatorics when we have exactly $d=5$ lines and $k=1$ conics, and we are able to show that only $11$ of them can be potentially realized geometrically, i.e., they satisfy Hirzebruch-type inequality, see Proposition \ref{cl}. We suspect that these remaining $11$ cases cannot be constructed in the complex plane, but so far we do not know how to prove it.
 From that perspective, it seems that in the class of conic-line arrangements we have just one maximizing curve that is constructed using Persson's triconical arrangement and one line \cite{DimcaPokoraMaximizing}. In Section \ref{sec4} we recall the construction of the only known maximizing conic-line arrangement of degree $7$, and the we present new examples of conic-line arrangements of degree $7$ that are free but not maximizing 

In the paper we work over the complex numbers and our symbolic computations are preformed using \verb}SINGULAR} \cite{Singular}.

\section{Preliminaries}
Let $S := \mathbb{C}[x,y,z]$ be the graded ring of polynomials with complex coefficients, and for a homogeneous polynomial $f \in S$ let us denote by $J_{f}$ the Jacobian ideal associated with $f$, that is, the ideal of the form $J_{f} = \langle \partial_{x}\, f, \partial_{y} \, f, \partial_{z} \, f \rangle$. We start with two fundamental definitions.
\begin{definition}
Let $p$ be an isolated singularity of a polynomial $f\in \mathbb{C}[x,y]$. Since we can change the local coordinates, assume that $p=(0,0)$.
Furthermore, the number 
$$\mu_{p}=\dim_\mathbb{C}\left(\mathbb{C}[x,y] /\bigg\langle \frac{\partial f}{\partial x},\frac{\partial f}{\partial y} \bigg\rangle\right)$$
is called the (local) Milnor number of $f$ at $p$.

The number
$$\tau_{p}=\dim_\mathbb{C}\left(\mathbb{C}[ x,y] /\bigg\langle f,\frac{\partial f}{\partial x},\frac{\partial f}{\partial y}\bigg\rangle \right)$$
is called the (local) Tjurina number of $f$ at $p$.
\end{definition}

For a projective situation, with a point $p\in \mathbb{P}^{2}_{\mathbb{C}}$ and a homogeneous polynomial $f\in \mathbb{C}[x,y,z]$, we can take local affine coordinates such that $p=(0,0,1)$ and then the dehomogenization of $f$.\\
Finally, the total Tjurina number of a given reduced curve $C \subset \mathbb{P}^{2}_{\mathbb{C}}$ is defined as
$$\tau(C) = \sum_{p \in {\rm Sing}(C)} \tau_{p}.$$ 
Moreover, if $C \, : f=0$ is a reduced plane curve with only quasi-homogeneous singularities, then
$$\tau(C) = \sum_{p \in {\rm Sing}(C)} \tau_{p} = \sum_{p \in {\rm Sing}(C)} \mu_{p} = \mu(C),$$
which means that the total Tjurina number of $C$ is equal to the total Milnor number of $C$.

Next, we will need the following important invariant that is attached to the syzygies of $J_{f}$.
\begin{definition}
Consider the graded $S$-module of Jacobian syzygies of $f$, namely $$AR(f)=\{(a,b,c)\in S^3 : af_x+bf_y+cf_z=0\}.$$
The minimal degree of non-trivial Jacobian relations for $f$ is defined as
$${\rm mdr}(f):=\min\{r : AR(f)_r\neq (0)\}.$$ 
\end{definition}

Now we define the main object of our studies, namely we define free plane curves. We are going to do it using the language of the minimal degree of (non-trivial) Jacobian relations and the total Tjurina number of a given curve $C \, : f=0$ following the lines of \cite{duP}.
\begin{definition}[Freeness]
Let $C \, : f=0$ be a reduced curve in $\mathbb{P}^{2}_{\mathbb{C}}$ of degree $d$. Then the curve $C$ with $r:={\rm mdr}(f)\leq (d-1)/2$ is free if and only if
\begin{equation}
\label{duPles}
(d-1)^{2} - r(d-r-1) = \tau(C).
\end{equation}
\end{definition}
The above definition is going to be extensively used in the section where we provide new examples of free curves.

Now we pass to the notion of maximizing curves that was introduced for reduced plane curves of even degree having ${\rm ADE}$ singularities, and for the completeness of the note, let us recall the classification of ${\rm ADE}$ singularities for curves by presenting their local normal forms \cite{arnold}:
\begin{center}
\begin{tabular}{ll}
$A_{k}$ with $k\geq 1$ & $: \, x^{2}+y^{k+1}  = 0$, \\
$D_{k}$ with $k\geq 4$ & $: \, y^{2}x + x^{k-1}  = 0$,\\
$E_{6}$ & $: \, x^{3} + y^{4} = 0$, \\
$E_{7}$ & $: \, x^{3} + xy^{3} = 0$, \\
$E_{8}$ & $:\, x^{3}+y^{5} = 0$.
\end{tabular}
\end{center}
All reduced plane curves that admit only ${\rm ADE}$ singularities will be called $\textbf{simply singular}$. Now we can define the second main object of our studies.

\begin{definition}(\cite[1.6. Definition]{Persson})
We say that a reduced simply singular curve $C \, : f=0$ in $\mathbb{P}^{2}_{\mathbb{C}}$ of \textbf{even degree} $n=2m\geq 4$ is \textit{maximizing} if
$$\tau(C) = 3m(m-1)+1.$$
\end{definition}
Very recently, Dimca and Pokora proved in \cite{DimcaPokoraMaximizing} the following crucial theorem that provides a linkage between free and maximizing reduced simply singular curves. Recall that if $C :\, f=0$ is a (reduced) free curve in $\mathbb{P}^{2}_{\mathbb{C}}$ of degree $d$, then the exponents of $C$ are defined as the pair
$$(d_{1},d_{2}) = ({\rm mdr}(f), d-1-{\rm mdr}(f)).$$
\begin{theorem}[Dimca-Pokora]
\label{maxi}
Let $C$ be a plane curve of degree $n=2m \geq 4$ having only ${\rm ADE}$ singularities. Then $C$ is maximizing if and only if $C$ is a free curve with the exponents $(d_{1},d_{2}) = (m-1,m)$.
\end{theorem}
In the second part of their paper, Dimca and Pokora introduced the notion of maximizing curves in the odd-degree case.
\begin{definition}(\cite[Definition 5.2]{DimcaPokoraMaximizing})
\label{def1}
Let $C \, : f=0$ be a reduced simply singular curve in $\mathbb{P}^{2}_{\mathbb{C}}$ of \textbf{odd degree} $n=2m+1\geq 5$. We say that $C$ is a \textit{maximizing} curve if 
$$\tau(C) = 3m^{2}+1.$$
\end{definition}
It is worth recalling that maximizing curves of odd degree are automatically free \cite[Proposition 5.1]{DimcaPokoraMaximizing} and hence we have a complete picture saying that being maximizing (in every degree) is the same as being free provided that the exponents satisfy an additional condition.
\section{Constraints on the existence of maximizing curves}
\label{Sec3}
Here we are going to derive conditions on the existence of maximizing curves of degree $7$. Recall that in degree $5$ we have a complete classification \cite[Corollary 5.6]{DimcaPokoraMaximizing}.
Our first result is the following classification result.
\begin{theorem}
\label{char}
There does not exist any reduced plane curve of degree $7$ with singularities of types $A_{1}, A_{3}, A_{5}, A_{7}, D_{4}, D_{6}$ that is maximizing.
\end{theorem}
In order to show that result, we will use the following \cite[Theorem 2.1]{DimcaSernesi}.
\begin{theorem}[Dimca-Sernesi]
\label{sern}
Let $C \, : \, f = 0$ be a reduced curve of degree $d$ in $\mathbb{P}^{2}_{\mathbb{C}}$ having only quasi-homogeneous singularities. Then $${\rm mdr}(f) \geq \alpha_{C}\cdot d - 2,$$
where $\alpha_{C}$ denotes the Arnold exponent of $C$.
\end{theorem}
It is worth recalling that the Arnold exponent of a given reduced curve $C \subset \mathbb{P}^{2}_{\mathbb{C}}$ is defined as the minimum over all Arnold exponents of singular points $p$ in $C$. In modern language, the Arnold exponents of singular points are nothing else but the log canonical thresholds of singularities. Let us explain how to compute these numbers in our setting. Since ${\rm ADE}$ singularities are quasi-homogeneous \cite{arnold}, we can use the following pattern (cf. \cite[Formula 2.1]{DimcaSernesi}). 

Recall that the germ $(C,p)$ is weighted homogeneous of type $(w_{1},w_{2};1)$ with $0 < w_{j} \leq 1/2$ if there are local analytic coordinates $y_{1},y_{2}$ centered at $p=(0,0)$ and a polynomial $g(y_{1},y_{2})= \sum_{u,v} c_{u, v} y_{1}^{u} y_{2}^{v}$ with $c_{u,v} \in \mathbb{C}$, where the sum is over all pairs $(u,v) \in \mathbb{N}^{2}$ with $u w_{1} + v w_{2}=1$. Using this description, the log canonical threshold can be defined as $${\rm lct}_{p}(g) := w_{1}+w_{2}.$$
After such a preparation, we can present our proof.
\begin{proof}
Let $C$ be a maximizing curve of degree $d=7$ with prescribed above singularities. Recall that, according to Definition \ref{def1}, a reduced plane curve $C \, : f=0$ of degree $d=7$ is maximizing if $\tau(C) = 28$, which means that in that case ${\rm mdr}(f)=2$. Now we are going to find a lower bound on ${\rm mdr}(f)$. Since our curve admits only $A_{1}, A_{3}, A_{5}, A_{7}, D_{4}, D_{6}$, then the Arnold exponent $\alpha_{C}$ is equal to
$$\alpha_{C} = 3/5,$$
and this follows from the following calculations:
\begin{enumerate}
\item[1)] if $p$ is a node with the corresponding local normal form $g$, then ${\rm lct}_{p}(g) =1$,
\item[2)] if $p$ is a singularity of type $A_{2k+1}$ with $k \in \{1,2,3\}$ with the corresponding local normal form $g$, then ${\rm lct}_{p}(g) = \frac{k+2}{2k+2}$,
\item[3)] if $p$ is a singularity of type $D_{2k}$ with $k \in \{2,3\}$ with the corresponding local normal form $g$, then ${\rm lct}_{p}(g) = \frac{k}{2k-1}$.
\end{enumerate}
Using Theorem \ref{sern}, we have
$${\rm mdr}(f) \geq \frac{3}{5}\cdot 7 - 2 = \frac{11}{5} = 2.2,$$
and hence we have ${\rm mdr}(f) = 3$, a contradiction.
\end{proof}

Our result shows that it will not be easy to construct new examples of maximizing curves in degree $7$. In the next step, we want to focus on the scenario where our curves are conic-line arrangements admitting singularities of type $A_{1}, A_{3}, A_{5}, A_{7}, D_{4}, D_{6}$ and $D_{8}$. In this setting there is room for such curves that they admit ${\rm mdr}$ equal to $2$. In order to construct such a maximizing conic-line arrangement of degree $7$, we are going to find all possible weak-combinatorics that are indicated by the naive counts. Going into the precise considerations, let $\mathcal{CL} \subset \mathbb{P}^{2}_{\mathbb{C}}$ be an arrangement of $k\in \{1,2,3\}$ smooth conics and $d = 7-2k$ lines such that it admits $n_{2}$ nodes, $n_{3}$ ordinary triple points, $t_{3}$ tacnodes, $t_{5}$ singularities of type $A_{5}$, $t_{7}$ singularities of type $A_{7}$, $d_{6}$ singularities of type $D_{6}$, $d_{8}$ singularities of type $D_{8}$. Assume that $\mathcal{CL}$ is maximizing, then the following system of equations holds:

$$
(\square):
\begin{cases}
n_{2} + 3t_{3} + 4n_{3} + 5t_{5} + 6d_{6} + 7t_{7} + 8d_{8} = \tau(\mathcal{C}) = 28,\\
n_{2} + 2t_{3} + 3n_{3} + 3t_{5} + 4t_{7} + 4d_{6} + 5d_{8} = \binom{7}{2}-k = 21 - k.
\end{cases}
$$
For $k \in \{1,2,3\}$ we have altogether \textbf{296 solutions} in non-negative integers, which means that we have exactly $296$ (naive) weak combinatorics that might be possible to realize geometrically over the complex numbers. Let us focus on the case when $k=1$. Then we have exactly \textbf{103} possible weak-combinatorics, namely

\begin{multline*}
(\star) : \quad (n_{2},n_{3},t_{3},t_{5},t_{7},d_{6},d_{8}) \in \bigg\{ (0,4,0,0,0,2,0),(0,4,2,0,0,1,0),(0,4,4,0,0,0,0), \\ (0,5,0,0,0,0,1), (0,5,1,1,0,0,0), (1,3,1,0,0,2,0), (1,3,3,0,0,1,0), (1,3,5,0,0,0,0), \\ (1,4,0,1,0,1,0), (1,4,1,0,0,0,1), (1,4,2,1,0,0,0), (1,5,0,0,1,0,0),(2,2,0,0,0,3,0), \\ (2,2,2,0,0,2,0), (2,2,4,0,0,1,0), (2,2,6,0,0,0,0), (2,3,0,0,0,1,1),
(2,3,1,1,0,1,0), \\ (2,3,2,0,0,0,1), (2,3,3,1,0,0,0), (2,4,0,2,0,0,0), (2,4,1,0,1,0,0),(3,1,1,0,0,3,0), \\ (3,1,3,0,0,2,0),(3,1,5,0,0,1,0),(3,1,7,0,0,0,0),(3,2,0,1,0,2,0),(3,2,1,0,0,1,1), \\ (3,2,2,1,0,1,0),(3,2,3,0,0,0,1),(3,2,4,1,0,0,0),(3,3,0,0,1,1,0),(3,3,0,1,0,0,1), \\ (3,3,1,2,0,0,0),(3,3,2,0,1,0,0),(4,0,0,0,0,4,0),(4,0,2,0,0,3,0),(4,0,4,0,0,2,0), \\ (4,0,6,0,0,1,0),(4,0,8,0,0,0,0),(4,1,0,0,0,2,1),(4,1,1,1,0,2,0),(4,1,2,0,0,1,1), \\ (4,1,3,1,0,1,0),(4,1,4,0,0,0,1),(4,1,5,1,0,0,0),(4,2,0,0,0,0,2),(4,2,0,2,0,1,0), \\ (4,2,1,0,1,1,0),(4,2,1,1,0,0,1),(4,2,2,2,0,0,0),(4,2,3,0,1,0,0),
(4,3,0,1,1,0,0), \\ (5,0,0,1,0,3,0),(5,0,1,0,0,2,1),(5,0,2,1,0,2,0),(5,0,3,0,0,1,1),(5,0,4,1,0,1,0), \\ (5,0,5,0,0,0,1),(5,0,6,1,0,0,0),(5,1,0,0,1,2,0),(5,1,0,1,0,1,1),(5,1,1,0,0,0,2), \\ (5,1,1,2,0,1,0),(5,1,2,0,1,1,0),(5,1,2,1,0,0,1),(5,1,3,2,0,0,0),
(5,1,4,0,1,0,0), \\ (5,2,0,0,1,0,1),(5,2,0,3,0,0,0),(5,2,1,1,1,0,0),(6,0,0,0,0,1,2),(6,0,0,2,0,2,0), \\ (6,0,1,0,1,2,0),(6,0,1,1,0,1,1),(6,0,2,0,0,0,2),(6,0,2,2,0,1,0),(6,0,3,0,1,1,0), \\ 
\end{multline*}
\begin{multline*}
(6,0,3,1,0,0,1),(6,0,4,2,0,0,0),(6,0,5,0,1,0,0),(6,1,0,1,1,1,0),(6,1,0,2,0,0,1), \\ (6,1,1,0,1,0,1),(6,1,1,3,0,0,0),(6,1,2,1,1,0,0),(6,2,0,0,2,0,0),(7,0,0,0,1,1,1), \\ (7,0,0,1,0,0,2),(7,0,0,3,0,1,0),(7,0,1,1,1,1,0),(7,0,1,2,0,0,1),(7,0,2,0,1,0,1), \\ (7,0,2,3,0,0,0),(7,0,3,1,1,0,0),(7,1,0,2,1,0,0),(7,1,1,0,2,0,0),(8,0,0,0,2,1,0), \\ (8,0,0,1,1,0,1),(8,0,0,4,0,0,0),(8,0,1,2,1,0,0),(8,0,2,0,2,0,0),(9,0,0,1,2,0,0) \bigg\}.  
\end{multline*}
Using Theorem A.1 from the Appendix, which is a Hirzebruch-type inequality proved for our conic-line arrangements, we can verify the following claim.
\begin{proposition}
\label{cl}
Among all weak combinatorics presented in $(\star)$, only the following $11$ weak-combinatorics satisfy \eqref{hib} in Theorem A.1, namely
\begin{multline*}
(\triangle) : \quad (n_{2},n_{3},t_{3},t_{5},t_{7},d_{6},d_{8}) \in \bigg\{(0,4,0,0,0,2,0), (0,5,0,0,0,0,1), (1,3,1,0,0,2,0), \\
(1,4,0,1,0,1,0), (1,4,1,0,0,0,1), (1,5,0,0,1,0,0), (2,2,0,0,0,3,0), \\
(2,3,0,0,0,1,1), (4,0,0,0,0,4,0), (4,1,0,0,0,2,1), (4,2,0,0,0,0,2) \bigg\}.
\end{multline*}
\end{proposition}
To find a new example of maximizing septics, we need to decide whether we can construct geometrically any of the above (naively found) weak combinatorics, but so far we do not know how to do this.
\section{Examples of free conic-line arrangements in degree $7$}
\label{sec4}
We start with the only known example of a maximizing conic-line arrangement. This example is somehow close with respect to our setting and that's why it fits into our discussion here.
\begin{example}[Maximizing septic]
Consider $\mathcal{C} = \{C_{1}, C_{2}, C_{3}\} \subset \mathbb{P}^{2}_{\mathbb{C}}$ with the defining equation
$$F(x,y,z) = (x^{2}+y^{2}-z^{2})(2x^2 + y^2 +2xz)(2x^2 + y^2 - 2xz) =0.$$
This arrangement is the aforementioned Persson's triconical arrangement \cite[2.1 Proposition]{Persson}. We have the following list of singular points:
$$n_{2} = 2, \quad t_{3}=1, \quad t_{7}=2,$$
and the arrangement $\mathcal{C}$ is free. Consider now the arrangement $\mathcal{CL}$ that is given by
$$G(x,y,z) = y(x^{2}+y^{2}-z^{2})(2x^2 + y^2 +2xz)(2x^2 + y^2 - 2xz).$$
Observe that the resulting arrangement admits $2$ singularities of type $D_{10}$, one singularity of type $D_{6}$, and two nodes. Since ${\rm mdr}(G)=2$ hence $\mathcal{CL}$ is maximizing.
\end{example}

Next, we provide new examples of free (but not maximizing) conic-line arrangements with some ${\rm ADE}$ singularities. The presented constructions come from \cite{Pelka}, where the author study smooth $\mathbb{Q}$-homology planes from the perspective of conic-line arrangements. Somehow surprisingly, some of examples of conic-line arrangements presented there are free curve. We will use the enumeration of configurations captured from paper \cite{Pelka}.

\begin{example}[Configuration 5] Consider the arrangements $\mathcal{CL}_{1}$ of $k=2$ conics and $3$ lines given by
$$Q_{1}(x,y,z) = x(y-z)(x-y)(y^{2}-xz)(y^2-xz-z^2).$$
We can directly check that the curve $\mathcal{CL}_{3}$ has three singularities of type $A_{1}$, two singularities of type $D_{4}$, one singularity of type $D_{6}$, one singularity of type $D_{10}$, and this gives us that $\tau(\mathcal{CL}_{1}) = 27$. Moreover, we can check, using \verb}SINGULAR}, that $r={\rm mdr}(Q_{1})=3$. Since
$$27 = (d-1)^2-r(d-r-1) = \tau(\mathcal{CL}_{3}),$$
we have that $\mathcal{CL}_{1}$ is free.
\end{example}

\begin{example}[Configuration 8] Consider the arrangements $\mathcal{CL}_{2}$ of $k=2$ conics and $3$ lines given by
$$Q_{2}(x,y,z) = x(x-y)((2\iota -2)x + 2\iota y + z)(y^{2} - xz)(z^{2}-4(xz-y^{2})),$$
where $\iota^2 +1=0$. We can directly check that the curve $\mathcal{CL}_{2}$ has four singularities of type $A_{1}$, one singularity of type $D_{4}$, two singularities of type $D_{6}$, and one singularity of type $A_{7}$, and this gives us that $\tau(\mathcal{CL}_{2}) = 27$. Moreover, we can check, using \verb}SINGULAR}, that $r={\rm mdr}(Q_{2})=3$. Since
$$27 = (d-1)^2-r(d-r-1) = \tau(\mathcal{CL}_{2}),$$
it means that $\mathcal{CL}_{2}$ is free.
\end{example}
\begin{example}[Configuration 12] Consider the arrangements $\mathcal{CL}_{3}$ of $k=2$ conics and $3$ lines given by
$$Q_{3}(x,y,z) = xz(2y-x-z)(y^{2} - xz)\bigg((-11-5\sqrt{5})z(2y-x-z)-2x(2y-z)\bigg).$$
We can directly check that the curve $\mathcal{CL}_{2}$ has two singularities of type $A_{1}$, three singularities of type $A_{3}$, one singularity of type $D_{4}$, two singularities of type $D_{6}$, and this gives us that $\tau(\mathcal{CL}_{3}) = 27$. Moreover, we can check, using \verb}SINGULAR}, that $r={\rm mdr}(Q_{3})=3$. Since
$$27 = (d-1)^2-r(d-r-1) = \tau(\mathcal{CL}_{3}),$$
we have that $\mathcal{CL}_{3}$ is free.
\end{example}

\begin{example}[Configuration 19] Consider the arrangements $\mathcal{CL}_{4}$ of $k=1$ conics and $5$ lines given by
$$Q_{4}(x,y,z) = x(y-3x-8z)(y+3x-8z)(y+x-4z)(y+9x+4z)(3x^{2}+y^{2}-16z^{2}).$$
We can directly check that the curve $\mathcal{CL}_{4}$ has two singularities of type $A_{1}$, one singularity of type $A_{3}$, four singularities of type $D_{4}$, one singularity of type $D_{6}$, and this gives us that $\tau(\mathcal{CL}_{4}) = 27$. Moreover, we can check, using \verb}SINGULAR}, that $r={\rm mdr}(Q_{4})=3$. Since
$$27 = (d-1)^2-r(d-r-1) = \tau(\mathcal{CL}_{4}),$$
we have that $\mathcal{CL}_{4}$ is free.
\end{example}

Now we pass to our constructions of free conic-line arrangements of degree $7$ that we have found during preparations of this note.

\begin{example}[Configuration A] Consider the arrangements $\mathcal{IC}_{1}$ of $k=1$ conics and $5$ lines given by
$$Q_{5}(x,y,z) = xy(x-z)(x+z)(y-z)(x^{2}+y^{2}-z^{2}).$$
We can directly check that the curve $\mathcal{IC}_{1}$ has five singularities of type $A_{1}$, one singularity of type $D_{4}$, three singularities of type $D_{6}$, and this gives us that $\tau(\mathcal{IC}_{1}) = 27$. Moreover, we can check, using \verb}SINGULAR}, that $r={\rm mdr}(Q_{5})=3$. Since
$$27 = (d-1)^2-r(d-r-1) = \tau(\mathcal{IC}_{1}),$$
we have that $\mathcal{IC}_{1}$ is free.
\end{example}

\begin{example}[Configuration B] Consider the arrangements $\mathcal{IC}_{2}$ of $k=1$ conics and $5$ lines given by
$$Q_{6}(x,y,z) = x(x-z)(x+z)(y-z)(y+z)(x^{2}+y^{2}-z^{2}).$$
We can directly check that the curve $\mathcal{IC}_{2}$ has five singularities of type $A_{1}$, two singularities of type $A_{3}$, one singularity of type $D_{4}$, two singularities of $D_{6}$, and this gives us that $\tau(\mathcal{IC}_{2}) = 27$. Moreover, we can check, using \verb}SINGULAR}, that $r={\rm mdr}(Q_{6})=3$. Since
$$27 = (d-1)^2-r(d-r-1) = \tau(\mathcal{IC}_{2}),$$
we have that $\mathcal{IC}_{2}$ is free.
\end{example}

\begin{example}[Configuration C] Consider the arrangements $\mathcal{IC}_{3}$ of $k=1$ conics and $5$ lines given by
$$Q_{7}(x,y,z) = (x-z)(x+z)(y+z)(x+6z/10)(y-x/2+z/2)(x^{2}+y^{2}-z^{2}).$$
We can directly check that the curve $\mathcal{IC}_{3}$ has three singularities of type $A_{1}$, two singularities of type $A_{3}$, three singularities of type $D_{4}$, one singularity of $D_{6}$, and this gives us that $\tau(\mathcal{IC}_{3}) = 27$. Moreover, we can check, using \verb}SINGULAR}, that $r={\rm mdr}(Q_{7})=3$. Since
$$27 = (d-1)^2-r(d-r-1) = \tau(\mathcal{IC}_{3}),$$
we have that $\mathcal{IC}_{3}$ is free.
\end{example}
\section*{\textbf{Appendix A} : Constraints on weak combinatorics of certain conic-line arrangements in the complex plane \\ by Piotr Pokora}
In this short section we provide a Hirzebruch-type inequality for conic line arrangements admitting singularities of types $A_{1}, A_{3}, A_{5}, A_{7}, D_{4}, D_{6}, D_{8}$. This result will be applied in Section \ref{Sec3} in order to exclude the existence of some possible weak combinatorics of maximizing curves of degree $7$.
\begin{theoremA}
\label{Hirzeb}
Let $\mathcal{C} \subset \mathbb{P}^{2}_{\mathbb{C}}$ be an arrangement consisting of $k\geq 1$ smooth conics and $d\geq 1$ lines with ${\rm deg}(\mathcal{C}) = 2k+d \geq 6$. Assume that $\mathcal{C}$ admits $n_{2}$ nodes, $n_{3}$ ordinary triple points, $t_{3}$ tacnodes, $t_{5}$ singularities of type $A_{5}$, $t_{7}$ singularities of type $A_{7}$, $d_{6}$ singularities of type $D_{6}$, and $d_{8}$ singularities of type $D_{8}$.
Then one has
\begin{equation}
\label{hib}
8k + n_{2} + \frac{3}{4}n_{3} \geq d + \frac{5}{2}t_{3} + 5t_{5} + \frac{29}{4}t_{7} +\frac{13}{8} d_{6} + \frac{15}{4} d_{8}.
\end{equation}
\end{theoremA}
\begin{proof}
Let $\mathcal{C} = \{\ell_{1}, ..., \ell_{d}, C_{1}, ..., C_{k}\} \subset \mathbb{P}^{2}_{\mathbb{C}}$ be our arrangement and denote by $C = \ell_{1} + ... + \ell_{d} + C_{1} + ... + C_{k}$ the associated divisor. The proof is very similar as in \cite{JPZ} - we work with the pair $(\mathbb{P}^{2}_{\mathbb{C}}, \frac{1}{2}C)$, and we apply the orbifold BMY-inequality \cite{Langer}, namely
$$(\star\star) \, : \quad \sum_{p \in {\rm Sing}(C)} 3\bigg( \frac{1}{2}(\mu_{p}-1)+1-e_{orb}(p;\mathbb{P}^{2}_{\mathbb{C}}, \frac{1}{2} C)\bigg)\leq  \frac{5}{4}(2k+d)^{2}- \frac{3}{2}(2k+d),$$
where $e_{orb}(p;\mathbb{P}^{2}_{\mathbb{C}},\alpha C)$ is the local orbifold Euler number of a given singularity $p$ and $\mu_{p}$ denotes the local Milnor number of $p$.
Using results devoted to local orbifold Euler numbers from \cite{Langer}, we can estimate the left-hand side of $(\star\star)$ which is equal to
$$\frac{9}{4}n_{2} + \frac{45}{8}t_{2}+\frac{117}{16}n_{3}  + \frac{35}{4}t_{5} + \frac{333}{32}d_{6} + \frac{189}{16}t_{7} + \frac{215}{16}d_{8}.$$
Then using a naive combinatorial count
$$(2k+d)^{2} = d + 4k + 2(n_{2}+3n_{3}+2t_{3} + 3t_{5} + 4t_{7} + 4d_{6} + 5d_{8})$$
and plugging the collected data into the BMY-inequality, we obtain the desired result.
\end{proof}

\vskip 0.5 cm

\bigskip
Izabela Czarnota,
Department of Mathematics,
Pedagogical University of Krakow,
Podchor\c a\.zych 2,
PL-30-084 Krak\'ow, Poland. \\
\nopagebreak
\textit{E-mail address:} \texttt{izabela.czarnota1@student.up.krakow.pl}
\bigskip
\end{document}